\documentclass[11pt, oneside,american,notitlepage]{article}
\usepackage[a4paper]{geometry}

\usepackage[utf8]{inputenc}

\usepackage{amssymb,amsmath,mathtools,bbm,xfrac,array}

\usepackage{amsthm}
\begingroup
    \makeatletter
    \@for\theoremstyle:=definition,remark,plain\do{%
        \expandafter\g@addto@macro\csname th@\theoremstyle\endcsname{%
            \addtolength\thm@preskip\parskip
            }%
        }
\endgroup

\usepackage[style=alphabetic,backend=biber,bibencoding=utf8,url=false,isbn=false,doi=false]{biblatex}

\usepackage{authblk}

\usepackage{parskip}

\usepackage{subfiles}

\usepackage[svgnames]{xcolor}

\usepackage[hidelinks]{hyperref}

\numberwithin{equation}{section}
\newtheorem{theorem}{Theorem}[section]
\newtheorem{proposition}[theorem]{Proposition}
\newtheorem{lemma}[theorem]{Lemma}

\newtheorem{corollary}[theorem]{Corollary}
\theoremstyle{definition}

\theoremstyle{remark}
\newtheorem{remark}[theorem]{Remark}
\newtheorem{example}[theorem]{Example}

\DeclareMathOperator{\Cov}{Cov}

\DeclareMathOperator{\id}{id}
\DeclareMathOperator{\Id}{Id}

\newcommand{\E}{\mathbb E}

\newcommand{\PP}{\mathbb P}

\newcommand{\QV}[1]{\langle#1\rangle}
\newcommand{\abs}[1]{\lvert#1\rvert}
\newcommand{\Abs}[1]{\left\lvert#1\right\rvert}
\newcommand{\norm}[1]{\lVert#1\rVert}

\newcommand{\lipnorm}[1]{\lVert#1\rVert_\mathrm{Lip}}
\newcommand{\R}{\mathbb{R}}

\newcommand{\todo}[1]{}

\title{Time Averages of Markov Processes and Applications to
  Two-Timescale Problems}
\author{Bob Pepin \thanks{bob.pepin@uni.lu or bobpepin@gmail.com}}
\affil{Mathematics Research Unit, FSTC, University of Luxembourg \\
Maison du Nombre, 4364 Esch-sur-Alzette, Grand-Duchy of Luxembourg}

\usepackage[showzone=false,showseconds=false,useregional]{datetime2}
\DTMsettimestyle{iso}

\begin{document}
\maketitle

\begin{abstract}
  We show a decomposition into the sum of a martingale and a
  deterministic quantity for time averages of the solutions to
  non-autonomous SDEs and for discrete-time Markov processes. In the
  SDE case the martingale has an explicit representation in terms of
  the gradient of the associated semigroup or transition operator. We
  show how the results can be used to obtain quenched Gaussian
  concentration inequalities for time averages and to provide deeper
  insights into Averaging principles for two-timescale processes.
\end{abstract}

\section{Introduction}
For a Markov process $(X_t)_t$ with $t \in [0, T]$ or $t =
0,1,\ldots,T$ let
\begin{equation*}
  S_T f = \int_0^T f(t, X_t) dt
\end{equation*}
in the continuous-time case or
\begin{equation*}
  S_T f = \sum_{t=0}^{T-1} f(t, X_t)
\end{equation*}
in discrete time.

In the first part of this work, we will show a decomposition of the
form
\begin{equation*}
  S_T f = \E S_T f + M^{T,f}_T
\end{equation*}
where $M^{T,f}$ is a martingale depending on $T$ and $f$ for which we
will give an explicit representation in terms of the transition
operator or semigroup associated to $X$. 

We then proceed to illustrate how the previous results can be used to obtain
Gaussian concentration inequalities for $S_T$ when $X$ is the solution
to an Itô SDE.

The last part of the work showcases a number of results on two-timescale
processes that follow from our martingale representation.

\section{Martingale Representation}

Consider the following SDE with time-dependent coefficients on $\R^n$:
\begin{align*}
  dX_t &= b(t, X_t)dt + \sigma(t, X_t) dB_t, X_0 = x
\end{align*}
where $B$ is a standard Brownian motion on $\R^n$ with filtration
$(\mathcal{F}_t)_{t\geq0}$ and $b(t, x), \sigma(t, x)$ are continuous
in $t$ and locally Lipschitz continuous in $x$. We assume that $X_t$
does not explode in finite time.

Denote $C_c^\infty$ the set of smooth compactly
supported space-time functions on $\R_+ \times \R^n$.

Let $P_{s,t}$ be the evolution operator associated to $X$,
\begin{equation*}
  P_{s,t} f(x) = \E\left[f(t, X_t) \middle| X_s = x\right], \quad f
  \in C_c^\infty.
\end{equation*}

For $T > 0$ fixed consider the
martingale
\begin{equation*}
  M_t = \E^{\mathcal{F}_t} \int_0^T f(s, X_s) ds.
\end{equation*}
and observe that since $X$ is adapted and by the Markov property
\begin{equation*}
  M_t = \int_0^t f(s, X_s) ds + \E^{\mathcal{F}_t} \int_t^T f(s, X_s)
  ds = \int_0^t f(s, X_s) ds + R_t^T f(X_t)
\end{equation*}
with
\begin{equation*}
  R_t^T f(x) = \int_t^T P_{t,s} f(x) ds.
\end{equation*}
By applying the Itô formula to $R_t^T f$ we can identify the
martingale $M$. This is the content of the following short theorem.

\begin{theorem}\label{theorem:M}
  For $T > 0$ fixed, $t \in [0, T]$ and $f \in C_c^\infty$
  \begin{equation*}
    \int_0^t f(s, X_s) ds + R_t^T f(X_t) = \E \int_0^T f(s, X_s) ds +
    M_t^{T,f}
  \end{equation*}
  with
  \begin{align*}
    M_t^{T,f} &= \int_0^t \nabla R_s^T f(X_s) \cdot \sigma(s, X_s) dB_s.
  \end{align*}

\end{theorem}
\begin{proof}
From the Kolmogorov backward equation $\partial_t
P_{t,s} f = -L_t P_{t,s} f$ and since $P_{t,t}f = f$ we have
\begin{align*}
  \partial_t R_t^T f(x) 
  &= -f(t, x) - \int_t^T L_t P_{t,s} f(x) ds = -f(t, x) - L_t R_t^T f(x).
\end{align*}
By Itô's formula
\begin{align*}
  R_t^T f(X_t) 
  &= R_0^T f(X_0) + \int_0^t \partial_s R_s^T f(X_s) ds + \int_0^t L_s R_s^T f(X_s) ds + 
    \int_0^t \nabla R_s^T f(X_s) \cdot \sigma(s, X_s) dB_s \\
  &= \E \int_0^T f(t, X_t) dt - \int_0^t f(s, X_s)ds + \int_0^t \nabla R_s^T f(X_s) \cdot \sigma(s, X_s) dB_s
\end{align*}
and we are done.


\end{proof}

\begin{remark}[Poisson Equation]
  In the time-homogeneous case $P_{t,s} = P_{s-t}$ and when the limit
  below is finite then it is independent of $t$ and we have
  \begin{equation*}
    R^\infty f := \lim_{T\to \infty} R_t^T f = \lim_{T\to \infty} \int_t^T
    P_{s-t} f ds 
    = \lim_{T\to \infty} \int_0^{T-t} P_s f ds = \int_0^\infty P_s f
    ds.
  \end{equation*}
  This is the resolvent formula for the solution to the Poisson
  equation $-Lg = f$ with $g = R^\infty f$.
\end{remark}

By taking $t = T$ in Theorem~\ref{theorem:M} we can identify the
martingale part in the martingale representation theorem for
$\int_0^T f(t, X_t)dt$.

\begin{corollary}\label{corollary:1} For $T > 0$ fixed, $f \in C_c^\infty$
  \begin{equation*}
    \int_0^T f(t, X_t) dt - \E \int_0^T f(t, X_t) dt =
    \int_0^T \nabla
    \int_t^T P_{t,s}f(X_t) ds \cdot \sigma(t, X_t) dB_t.
  \end{equation*}
\end{corollary}


  By applying the Itô formula to $P_{t,T} f(X_t)$ we obtain for
  $T > 0$ fixed
  \begin{equation}\label{eq:dptgrad}
    dP_{t,T} f(X_t) = \nabla P_{t,T} f(X_t) \cdot \sigma(t, X_t) dB_t
  \end{equation}
  and by integrating from $0$ to $T$
  \begin{equation*}
    f(T, X_T) = \E\left[f(T, X_T)\right] + \int_0^T
    \nabla P_{t,T} f(X_t) \cdot \sigma(t, X_t) dB_t.
  \end{equation*}
  This was observed at least as far back as~\autocite{elliott_integration_1989} and is
  commonly used in the derivation of probabilistic formulas for
  $\nabla P_{s,t}$.

Combining the formula~\eqref{eq:dptgrad} with
Theorem~\ref{theorem:M} we obtain the following expression for
$S_t - \E S_t$ in terms of $\nabla P_{s,t} f$.
\begin{corollary}
  For $f \in C_c^\infty$, $T > 0$ fixed and any $t < T$
  \begin{equation*}
    \int_0^t f(s, X_s) - \E f(s, X_s) ds = M_t^{T,f} - Z_t^{T,f}
  \end{equation*}
  with
  \begin{align*}
    Z_t^{T,f} &= \int_t^T \int_0^t \nabla P_{r,s} f(r, X_r) \cdot
                \sigma(r, X_r)\, dB_r \, ds \\
    M_t^{T,f} &= \int_0^t \int_r^T \nabla P_{r,s} f(r, X_r) \, ds \cdot
                \sigma(r, X_r) \, dB_r.
  \end{align*}
\end{corollary}
\begin{proof}
  Let $f_0(t, x) = f(t, x) - \E f(t, X_t) = f(t, x) - P_{0,t}(x_0)$. We have
  \begin{align*}
    R_t^T f_0(X_t) 
    &= \int_t^T P_{t,s} f_0(X_t) ds \\
    &= \int_t^T P_{t,s} f(X_t) - P_{0,s}f(X_0) ds \\
    &= \int_t^T \int_0^t \nabla P_{r,s} f(r, X_r) \cdot \sigma(r, X_r)
      dB_r \, ds
  \end{align*}
  where the last equality follows by integrating~\eqref{eq:dptgrad}
  from $0$ to $t$ (with $T = s$). Since $R_0^T f_0 = 0$ and
  $\nabla P_{t,s} f_0 = \nabla P_{t,s} f$ we get from
  Theorem~\ref{theorem:M} that
  \begin{equation*}
    \int_0^t f_0(s, X_s) ds = M^{T,f}_t - R_t^T f_0(X_t)
  \end{equation*}
  and the result follows with $Z^{T,f}_t = R_t^T f_0(X_t)$.
\end{proof}




\begin{remark}[Carré du Champs and Mixing]
For differentiable functions $f, g$ let
\begin{equation*}
  \Gamma_t(f, g)(x) = \tfrac12 \nabla f(t, x) (\sigma\sigma^\top)(t, x)
  \nabla g(t, x).
\end{equation*}
Then we have the following expression for the quadratic variation of $M^{T,f}$:
\begin{align*}
  d\QV{M^{T,f}}_t &=
              \Abs{\int_t^T \sigma(t, X_t)^\top \nabla P_{t,s} f(X_t)\,
              ds}^2 dt \\
            &= \left( 4 \int_{t \leq s \leq r \leq T} \Gamma_t(P_{t,s} f, P_{t,r}
              f)(X_t)\, dr \, ds \right) dt.
\end{align*}

Furthermore, since
\begin{equation*}
  \partial_s P_{r,s} (P_{s,t}f P_{s,t}g) = 2
  P_{r,s}(\Gamma_s(P_{s,t}f, P_{s,t}g))
\end{equation*}
and setting $g(t, x) = \int_t^T P_{t, s} f(x) ds$
we have
\begin{align*}
  E\QV{M^{T,f}}_t 
  &= 2 \int_0^T \int_t^T 2 P_{0,t} \Gamma_t(P_{t,s}f, P_{t,s}g) ds \, dt \\
  &= 2 \int_0^T \int_t^T \partial_t P_{0,t}(P_{t,s}f P_{t,s}g) ds \, dt
  \\
  &= 2 \int_0^T \partial_t \int_t^T P_{0,t}(P_{t,s}f P_{t,s}g) ds\, dt +
    2 \int_0^T P_{0,t}(f g) dt \\
  &= 2 \int_0^T P_{0,t}(fg) - P_{0,t}f P_{0,t}g \, dt \\
  &= 2 \int_{0 \leq t \leq s \leq T} \Cov(f(t, X_t), f(s, X_s)) ds \, dt.
\end{align*}
This shows how the expressions we obtain in terms of the gradient of
the semigroup relate to mixing properties of $X$.
\end{remark}

\begin{remark}[Pathwise estimates]


We would like to have a similar estimate for
\begin{equation*}
  \E \sup_{0\leq t\leq T} \Abs{\int_0^t f(X_s) - \E f(X_s) ds}.
\end{equation*}

Setting
\begin{equation*}
f_0(t, x) = f(x) - \E f(X_t) = f(x) - P_{0,t} f(x_0)
\end{equation*}
we have
\begin{align*}
  \E \sup_{0\leq t\leq T} \Abs{\int_0^t f(X_s) - \E f(X_s) ds} 
  & \leq \E \sup_{0\leq t\leq T} \abs{M_t^{T,f_0}} + \E \sup_{0\leq
    t\leq T} \abs{R_t^T f_0(X_t)} \\
  & \leq 2\left(\E\QV{M^{T, f_0}}_T\right)^{1/2} +  \E \sup_{0\leq
    t\leq T} \abs{R_t^T f_0(X_t)}
\end{align*}
and
\begin{align*}
  R_t^T f_0 (X_t)
  &= \int_t^T P_{t,s} f(X_t) - P_{0,s} f(x_0) ds \\
  &= \int_t^T \int_0^t \nabla P_{r,s} f(X_r) \cdot \sigma(r, X_r) dB_r ds
\end{align*}
where the last equality follows from (for $s$ fixed)
\begin{equation*}
  dP_{t,s}f(X_t) = \nabla P_{t,s} f(X_t) \cdot \sigma(t, X_t) dB_t.
\end{equation*}
\end{remark}

\subsection{Discrete time}
Consider a discrete-time Markov process $(X_n)_{n=1\ldots N}$ with
transition operator
\begin{equation*}
  P_{m,n} f(x) = \E[f_n(X_n) | X_m = x]
\end{equation*}
and generator
\begin{equation*}
  L_n f(x) = P_{n,n+1} f(x) - f_{n}(x).
\end{equation*}
As in the continuous-time setting
\begin{equation*}
  M_n := f_n(X_n) - f_0(X_0) - \sum_{m=0}^{n-1} L_m f(X_m)
\end{equation*}
is a martingale (by the definition of $L$) and by direct calculation
\begin{equation*}
  M_n - M_{n-1} = f_n(X_n) - P_{n-1,n} f(X_{n-1}).
\end{equation*}
Let
\begin{equation*}
  R_n^N f(x) = \sum_{m=n}^{N-1} P_{n,m} f(x)
\end{equation*}
and observe that
\begin{equation*}
  L_n R^N f(x) = \sum_{m=n+1}^{N} P_{n,n+1}P_{n+1,m} f(x) -
  \sum_{m=n}^{N-1} P_{n,m} f(x) = -f_{n}(x).
\end{equation*}
Note that
\begin{equation*}
  R_N^N f(x) = 0 \text{ and } R_0^N f(x) = \E\left[\sum_{m=n}^{N-1} f(X_m)
  \middle| X_0 = x\right].
\end{equation*}
It follows that
\begin{equation*}
  \sum_{m=0}^{n-1} f_m(X_m) + R_n^N f(X_n) = -\sum_{m=0}^{n-1} L_m R^N
  f(X_m) + R_n^N f(X_n) = R_0^N f(X_0) + M_n^{N,f}
\end{equation*}
with
\begin{equation*}
  M^{N,f}_n - M^{N,f}_{n-1} = \sum_{m=n}^{N-1} P_{n,m} f(X_n) -
P_{n-1,m} f(X_{n-1}).
\end{equation*}
Analogous to the continuous-time case, we define the carré du champs
\begin{align*}
  \Gamma_n(f, g) &:= L_n (fg) - g_n L_n f - f_n L_n g \\
                 &= P_{n,n+1}(fg) - f_n P_{n,n+1} g - g_n P_{n,n+1} f
                   + f_n g_n
  \\
                 &= \E\left[(f_{n+1}(X_{n+1}) - f_n(X_n))(g_{n+1}(X_{n+1}) - g_n(X_n)) | {\mathcal{F}_n}\right]
\end{align*}
and using the summation by parts formula
\begin{align*}
\lefteqn{  \QV{M^{N,f}}_n - \QV{M^{N,f}}_{n-1} 
  = \E[(M^{N,f}_n - M^{N,f}_{n-1})^2 | \mathcal{F}_{n-1}]}\quad \\
  & = 2\sum_{n\leq k \leq m < N-1} \E\left[(P_{n,m}f(X_{n}) -
    P_{n-1,m}f(X_{n-1})) (P_{n,k}f(X_{n}) - P_{n-1,k}f(X_{n-1})) |
    {\mathcal{F}_{n-1}}\right] \\ & \quad \quad + \sum_{m=n}^{N-1} \E\left[(P_{n,m}f(X_{n}) -
    P_{n-1,m}f(X_{n-1}))^2 |
    {\mathcal{F}_{n-1}}\right] \\
  &= 2\sum_{m=n}^{N-1} \sum_{k=m}^{N-1} \Gamma_{n-1}(P_{n-1,m} f, P_{n-1,k}
    f)(X_{n-1}) + \sum_{m=n}^{N-1} \Gamma_{n-1}(P_{n-1,m}f)(X_{n-1}).
\end{align*}

\section{Concentration inequalities from exponential gradient bounds}

In this section we focus on the case where we have uniform exponential decay of 
$\nabla P_{s,t}$ so that
\begin{equation}\label{eq:gradexp}
\abs{\sigma(s, x)^\top \nabla P_{s,t} f(x)} \leq C_s e^{-\lambda_s
  (t-s)} \quad (0 \leq s \leq t \leq T)
\end{equation}
for all $x \in \R^n$ and some class of functions $f$.

We first show that exponential gradient decay implies a
concentration inequality.

\begin{proposition}\label{prop:gaussian} For $T > 0$ fixed and all
  functions $f$ such that~\eqref{eq:gradexp} holds we have
  \begin{equation*}
    \PP\left(\frac{1}{T}\int_0^T f(t, X_t) - \E f(t, X_t) dt > R\right) \leq
        e^{-\frac{R^2 T}{V_T}}, \quad V_T = \frac1T \int_0^T
        \left({\frac{C_t}{\lambda_t} \left(1 - e^{-\lambda_t (T-t)}\right)}\right)^2 dt
  \end{equation*}
\end{proposition}
\begin{proof} By~\eqref{eq:gradexp}
\begin{align*}
  d\QV{M^{T, f}}_t 
  &= \Abs{ \int_t^T \sigma(t, X_t)^\top \nabla P_{t,s} f(X_t) ds}^2 dt \\ 
  &\leq {\left(\int_t^T C_t e^{-\lambda_t (s-t)} ds\right)}^2 dt
  = {\left(\frac{C_t}{\lambda_t} \left(1 - e^{-\lambda_t (T-t)}\right)\right)}^2 dt
\end{align*}
so that $\QV{M^{T, f}}_T \leq V_T T$.

By Corollary~\ref{corollary:1} and since Novikov's condition holds
trivially due to $\QV{M^{T,f}}$ being bounded by a deterministic
function we get
  \begin{multline*}
    \E \exp\left(a \int_0^T f(t, X_t) - \E f(t, X_t)dt\right) = \E \exp\left(a M^{T,f}_T\right)
    \\ \leq \E \left[\exp\left(a M^{T,f}_T - \frac{a^2}{2}\QV{M^{T,f}}_T\right)\right]
    \exp\left(\frac{a^2}{2}\QV{M^{T,f}}_T\right) \leq
    \exp\left(\frac{a^2}{2} V_T T\right).
  \end{multline*}
  By Chebyshev's inequality
  \begin{equation*}
    \PP\left(\frac{1}{T}\int_0^T f(t, X_t) - \E f(t, X_t) dt > R\right) \leq
    \exp\left(-a R T\right) \exp\left(\frac{a^2}{2} V_T T\right)
  \end{equation*}
  and the result follows by optimising over $a$.
\end{proof}

The corresponding lower bound is obtained by replacing $f$ by $-f$.

For the rest of this section, suppose that $\sigma = \Id$ and that we are in the
time-homogeneous case so that $P_{s,t} = P_{t-s}$.
An important case where bounds of the form~\eqref{eq:gradexp} hold is
when there is exponential
contractivity in the $L^1$ Kantorovich (Wasserstein) distance $W_1$.
If for any two probability measures $\mu, \nu$ on $\R^n$
\begin{equation}\label{eq:W1exp}
  W_1(\mu P_t, \nu P_t) \leq C e^{-\lambda t} W_1(\mu, \nu).
\end{equation}
then~\eqref{eq:gradexp} holds for all
Lipschitz functions $f$ with $C_s = C$, $\lambda_s = \lambda$.

Here the distance $W_1$ between two probability measures
$\mu$ and $\nu$ on $\R^n$ is defined by
\begin{equation*}
  W_1(\mu, \nu) = \inf_\pi \int \abs{x-y} \pi(dx\, dy)
\end{equation*}
where the infimum runs over all couplings $\pi$ of $\mu$. We also have
the Kantorovich-Rubinstein duality
\begin{equation}\label{eq:kantorovich}
  W_1(\mu, \nu) = \sup_{\lipnorm{f}\leq1} \int f d\mu - \int f d\nu
\end{equation}
and we use the notation
\begin{equation*}
  \lipnorm{f} = \sup_{x \ne y} \frac{f(x) - f(y)}{\abs{x - y}}.
\end{equation*}

We can see that~\eqref{eq:W1exp} implies~\eqref{eq:gradexp} from
\begin{multline*}
  \abs{\nabla P_t f}(x) 
  = \lim_{y \to x} \frac{\abs{P_t f(y) - P_t f(x)}}{\abs{y-x}} 
  \leq \lim_{y \to x} \frac{W_1(\delta_{y} P_t, \delta_x
    P_t)}{\abs{y-x}} \\
  \leq \lipnorm{f} C e^{-\lambda
      t} \lim_{y \to x} \frac{W_1(\delta_{y}, \delta_x)}{\abs{y-x}}
    = \lipnorm{f} C e^{-\lambda
      t}
\end{multline*}
where the first inequality is due to the Kantorovich-Rubinstein
duality~\eqref{eq:kantorovich} and the second is~\eqref{eq:gradexp}. 

Bounds of the form~\eqref{eq:W1exp} have been obtained using coupling methods in
\autocite{eberle_reflection_2016,eberle_quantitative_2016,wang_exponential_2016,}
under the condition that there exist
positive constants $\kappa, R_0$ such that
\begin{equation*}
  (x-y) \cdot (b(x) - b(y)) \leq -\kappa \abs{x-y}^2 \text{ when }
  \abs{x - y} > R_0.
\end{equation*}
Similar techniques lead to the
corresponding results for kinetic Langevin diffusions\autocite{eberle_couplings_2017,}.

Using a different approach, in \cite{crisan_pointwise_2016} the
authors directly show uniform exponential contractivity of the
semigroup gradient for bounded continuous functions, focusing on
situations beyond hypoellipticity.

Besides gradient bounds, exponential contractivity in $W_1$ also
implies the existence of a stationary measure $\mu_\infty$
\autocite{eberle_reflection_2016}.
Proposition~\ref{prop:gaussian} now leads to a simple
proof of a deviation inequality that was obtained in a
similar setting in \autocite{joulin_new_2009,} via a tensorization argument.

\begin{proposition} If~\eqref{eq:W1exp} holds
  then for all Lipschitz
  functions $f$ and all initial measures $\mu_0$
  \begin{equation*}
    \PP_{\mu_0} \left(\frac{1}{T}\int_0^T f(X_t) dt - \int f d\mu_\infty >
    R\right) \leq \exp\left(-\left(\frac{\lambda
      \sqrt{T}\, R}{C \lipnorm{f} (1-e^{-\lambda T})} - \frac{W_1(\mu_0, \mu_\infty)}{\sqrt{T}}\right)^2\right)
  \end{equation*}
\end{proposition}
\begin{proof}
  We start by applying Proposition~\ref{prop:gaussian} so that
  \begin{align*}
    \lefteqn{\PP_{\mu_0} \left(\frac{1}{T}\int_0^T f(X_t) dt - \int f d\mu_\infty >
    R\right) }\quad \\
  &
    = \PP_{\mu_0} \left(\frac{1}{T}\int_0^T f(X_t) - \E f(X_t) dt > R +
    \frac{1}{T}\int_0^T \mu_\infty(f) - \mu_0 P_t (f) dt\right) \\
  & \leq \exp\left(-\left(R - \Abs{\frac{1}{T}\int_0^T \mu_\infty(f) - \mu_0 P_t (f)
    dt}\right)^2 \frac{T}{V_T}\right), \quad V_T = \left(\frac{\lipnorm{f}
    C (1-e^{-\lambda T})}{\lambda}\right)^2.
\end{align*}
By the Kantorovich-Rubinstein duality
\begin{align*}
  \Abs{\frac{1}{T}\int_0^T \mu_\infty(f) - \mu_0 P_t (f) dt} 
  &\leq \Abs{\frac{1}{T} \int_0^T \norm{\nabla f}_\infty W_1(\mu_\infty P_t,
    \mu_0 P_t) dt} \\
  &\leq \frac{\norm{\nabla f}_\infty C}{\lambda} \frac{(1 -
    e^{-\lambda T})}{T} W_1(\mu, \mu_0) = \frac{\sqrt{V_T}}{T} W_1(\mu, \mu_0).
\end{align*}
from which the result follows immediately.
\end{proof}

\section{Averaging: Two-timescale Ornstein-Uhlenbeck}
Consider the following linear multiscale SDE on $\R \times \R$
where the first component is accelerated by a factor $\alpha \geq 0$:
\begin{align*}
  dX_t &= -\alpha (X_t - Y_t) dt + \sqrt{\alpha} dB^X_t, \quad X_0 = x_0 \\
  dY_t &= -(Y_t - X_t) dt + dB^Y_t,\quad Y_0 = y_0
\end{align*}
with $B^X, B^Y$ independent Brownian motions on $\R$. Denote $P_t$ and
$L$ the associated semigroup and infinitesimal generator respectively.

Let $f(x, y) = x - y$ and note that $Lf = -(\alpha + 1) f$. We have by
the regularity of $P_t$ and the Kolmogorov forward equation
\begin{equation*}
  \partial_t \partial_x P_t f = \partial_x P_t L f = -(\alpha + 1)
  \partial_x P_t f
\end{equation*}
so that
\begin{equation*}
  \partial_x P_t f = \partial_x f e^{-(\alpha+1)t} = e^{-(\alpha+1)t}.
\end{equation*}
Repeating the same reasoning for $\partial_y P_t$ and $P_t$ gives
\begin{equation*}
  \partial_y P_t f = -e^{-(\alpha+1)t} \text{ and }
  P_t f(x, y) = (x-y) e^{-(\alpha+1) t}.
\end{equation*}
From Corollary~\ref{corollary:1}
\begin{equation*}
  \int_0^T X_t - Y_t \, dt = R_0^T f(x_0, y_0) + M_T^{T,f}
\end{equation*}
with
\begin{align*}
  R_t^T f(x, y) 
  &= \int_t^T P_{s-t} f(x, y) ds = (x-y)
                  \frac{1-e^{-(\alpha+1)(T-t)}}{\alpha+1}, \\
  M_T^{T,f} 
  &= \int_0^T \int_t^T \partial_x P_{s-t} f(X_t, Y_t) ds \, \sqrt{\alpha} dB^X_t
  + \int_0^T
  \int_t^T \partial_y P_{s-t} f(X_t, Y_t) ds\, dB^Y_t \\
  &= \int_0^T \frac{1 -
    e^{-(\alpha+1)(T-t)}}{\alpha+1} (\sqrt{\alpha} dB^X_t - dB^Y_t).
\end{align*}
This shows that for each $T$ fixed
\begin{equation*}
  Y_T - (B^Y_T + y_0) = \int_0^T X_t - Y_t dt 
\end{equation*}
is a Gaussian random variable with mean
\begin{equation*}
  R_0^T = (x_0-y_0)
                  \frac{1-e^{-(\alpha+1)T}}{\alpha+1}
\end{equation*}
and variance
\begin{equation*}
  \QV{M^{T,f}}_T = \frac{1}{(\alpha+1)}\int_0^T
  \left(1-e^{-(\alpha+1)(T-t)}\right)^2 dt.
\end{equation*}


\section{Averaging: Exact gradients in the linear case}
Consider
\begin{align*}
  dX_t &= -\alpha (X_t - Y_t) dt + \sqrt{\alpha} dB^X_t, \quad X_0 = x_0 \\
  dY_t &= -(Y_t - X_t) dt - \beta Y_t + dB^Y_t,\quad Y_0 = y_0
\end{align*}
Denote $Z_t((x, y)) = (X_t(x), Y_t(x))$ the solution for $X_0 =
x, Y_0 = y$ and
let $V_t(z, v) = Z_t(z+v) - Z_t(z)$. Then
\begin{equation*}
  dV_t = -A V_t \, dt \text{ with } A = \left(\begin{matrix} \alpha & -\alpha
    \\ -1 & (1+\beta) \end{matrix}\right).
\end{equation*} 

The solution to the linear ODE for $V_t$ is
\begin{equation*}
  V_t(z, v) = e^{-At} v
\end{equation*}
Since $V_t$ does not depend on $z$ we drop it from the notation.  Now
for any continuously differentiable function $f$ on $\R^2$ and
$v \in \R^2, z \in \R^2$ we obtain the following expression for the
gradient of $P_t f(z)$ in the direction $v$:
\begin{align*}
  \nabla_v P_t f(z) 
  &= \lim_{\varepsilon\to 0} \frac{P_t f(z+\varepsilon
  v) - P_t f(z)}{\varepsilon} = \lim_{\varepsilon\to 0} \frac{\E f(Z_t(z+\varepsilon
  v)) - f(Z_t(z))}{\varepsilon} \\
  &=  \lim_{\varepsilon\to 0} \frac{\E \nabla f(Z_t(z)) \cdot
    V_t(\varepsilon v) + o(\abs{V_t(\varepsilon v)})}{\varepsilon} \\
  &= \E \nabla f (Z_t(z)) \cdot e^{-At} v.
\end{align*}
Since
$\nabla_v P_t f = \nabla P_t f \cdot v$
we can identify $\nabla P_t f(z) = \E^z (e^{-At})^\top \nabla f(Z_t)$.

The eigenvalues of $A$ are $(\lambda_0, \alpha \lambda_1)$ with
\begin{align*}
  \lambda_0 &= \frac12 \left(\alpha + \beta + 1 - \sqrt{(\alpha +
              \beta + 1)^2 - 4 \alpha \beta} \right), \\
  \lambda_1 &= \frac1{2\alpha} \left(\alpha + \beta + 1 + \sqrt{(\alpha +
              \beta + 1)^2 - 4 \alpha \beta} \right).
\end{align*}

By observing that
\begin{equation*}
  (\alpha + \beta + 1)^2 - 4\alpha\beta = (\alpha - (1+\beta))^2 +
  4\alpha = (\beta - (\alpha + 1))^2 + 4\beta
\end{equation*}
we see that asymptotically as $\alpha \to \infty$
\begin{align*}
  \lambda_0 & = \beta +  O\left(\frac{1}{\alpha}\right) \\
  \lambda_1 &= 1 + \frac{1}{\alpha} + O\left(\frac{1}{\alpha^2}\right).
\end{align*}

We can compute the following explicit expression for $e^{-At}$
\begin{align*}
  e^{-At} &= c_0(t) \Id - \frac{c_1(t)}{\alpha} A \\
          &= \left(\begin{matrix}
        \frac{c_2(t)}{\alpha} & c_1(t) \\
        \frac{c_1(t)}{\alpha} & c_0(t) - \frac{1+\beta}{\alpha}c_1(t)
    \end{matrix}\right)
\end{align*}
with
\begin{align*}
  c_0(t) &= \frac{{\alpha\lambda_1} e^{-\lambda_0 t} - \lambda_0 e^{-{\alpha\lambda_1}
  t}}{{\alpha\lambda_1} - \lambda_0} = \frac{(1 + \alpha)e^{-\lambda_0 t} -
           \beta e^{-{\alpha\lambda_1}
           t}}{{\alpha\lambda_1} - \lambda_0} +  O\left(\frac{1}{\alpha^2}\right), \\ 
  c_1(t) &= \frac{\alpha}{{\alpha\lambda_1} -
           \lambda_0}\left(e^{-\lambda_0 t} -
  e^{-{\alpha\lambda_1} t}\right), \\
  c_2(t) & = \alpha(c_0(t) - c_1(t)) = \frac{\alpha}{{\alpha\lambda_1} -
           \lambda_0} \left(e^{-\lambda_0 t} - (\beta -
           \alpha)e^{-{\alpha\lambda_1} t}\right) +  O\left(\frac{1}{\alpha}\right).
\end{align*}
Note that $\lambda_0, \lambda_1, c_0, c_1$ and $c_2$ are all of order
$O(1)$ as $\alpha \to \infty$.

We obtain
\begin{align*}
  \sigma^\top \nabla P_t f(z) 
  &= \E\left[ \left(\begin{matrix}
        \frac{c_2(t)}{\sqrt{\alpha}} & \frac{c_1(t)}{\sqrt{\alpha}} \\
        c_1(t) & c_0(t) - \frac{1+\beta}{\alpha}c_1(t)
      \end{matrix}\right)\nabla f(Z_t) \right] \\
  &= \frac{\alpha}{1+\alpha} \left(G_0 e^{-\lambda_0 t} +
    G_1\alpha e^{-{\alpha\lambda_1} t}\right) P_t{\nabla
    f(z)}
\end{align*}
with
\begin{align*}
  G_0 &= \left(\begin{matrix}
        \frac{1}{\sqrt{\alpha}} & \frac1{\sqrt{\alpha}} \\ 
        1 & 1
      \end{matrix}\right) + O\left(\frac{1}{\alpha}\right) \\
  G_1 &= \left(\begin{matrix}
        \frac{1}{\sqrt{\alpha}} -
        \frac{\lambda_0}{\alpha\sqrt{\alpha}} & -\frac1{\alpha \sqrt{\alpha}} \\ 
        -\frac{1}{\alpha} & -\frac{1 + \lambda_0 + \beta}{\alpha^2}
      \end{matrix}\right) = \left(\begin{matrix}
        \frac{1}{\sqrt{\alpha}} & 0 \\ 0 & 0
    \end{matrix}\right) + O\left(\frac{1}{\alpha}\right)
\end{align*}
The expression for $G_0$ shows that
$\abs{\sigma^\top \nabla P_t f(z)}$ can be of order $1/\sqrt{\alpha}$ only
for functions $f_\alpha(z)$ such that
$\E^z[ \partial_x f_\alpha(Z_t) + \partial_y f_\alpha(Z_t) ] =
O(1/\sqrt{\alpha})$.

Furthermore, for any function $f \in C_c^\infty$ we have
\begin{equation*}
  \Cov\left(f(Z_t), B^X_t\right) = O\left(\frac{1}{\sqrt{\alpha}}\right)
\end{equation*}
and
\begin{equation*}
  \Cov\left(\int_0^t f(s, Z_s) ds, B^X_t\right) = O\left(\frac{1}{\sqrt{\alpha}}\right).
\end{equation*}
Indeed, since $dP_{s,t} f(Z_s) = f(Z_s)\cdot \sigma dB_s$ we have
\begin{align*}
  f(Z_t) - \E f(Z_t) 
  &= \int_0^t \nabla P_{s,t} f(Z_s) \cdot \sigma dB_s \\
  &= \int_0^t \nabla_x P_{s,t} f(Z_s) \sqrt{\alpha} dB^X_s 
    + \int_0^t \nabla_y P_{s,t} f(Z_s) dB^Y_s 
\end{align*}
we have
\begin{align*}
  \Cov\left(f(Z_t), B^X_t\right) 
  &= \E\left[(f(Z_t) - \E f(Z_t)) B^X_t\right] \\
  &= \E\left[\int_0^t \nabla_x P_{s,t} f(Z_s) \sqrt{\alpha} ds\right] \\
  &= \left(\frac{1}{\sqrt{\alpha}} +
    O\left(\frac{1}{\alpha}\right) \right) \frac{\alpha}{1+\alpha} \int_0^t e^{-\lambda_0 s}
    P_{s,t}(\nabla_x f + \nabla_y f)(Z_s) ds.
\end{align*}
The result for $\int_0^t f(s, Z_s) ds$ follows by the same arguments
from the martingale representation for
$\int_0^t f(s, Z_s) ds - \E \int_0^t f(s, Z_s) ds$.

\section{Averaging: Conditioning on the slow component}

Consider the following linear multiscale SDE on $\R \times \R$
accelerated by a factor $\alpha$:
\begin{align*}
  dX_t &= -\alpha \kappa_X (X_t - Y_t) dt + \sqrt{\alpha} \sigma_X dB^X_t, 
         \quad X_0 = 0 \\
  dY_t &= -\kappa_Y (Y_t - X_t) dt + \sigma_Y dB^Y_t, \quad Y_0 = 0
\end{align*}
where $B^X, B^Y$ are independent Brownian motions and $\alpha, \kappa_X,
\kappa_Y, \sigma_X, \sigma_Y$ are strictly positive constants and we
are interested in the solution on a fixed inverval $[0, T]$.

We define the corresponding averaged process to be the solution to
\begin{subequations}
  \begin{align}
    d\bar{X}_t &= -\alpha \kappa_X(\bar{X}_t - \bar{Y}_t)dt + \sqrt\alpha
                 \sigma_X dB^X_t, \quad \bar{X}_0 = 0 \label{eq:Xbar} \\
    d\bar{Y}_t &= \E\left[-\kappa_Y (\bar{Y}_t - \bar{X}_t) \middle|
                 \mathcal{F}^{\bar{Y}}_t \right]dt +
                 \sigma_Y dB^Y_t, \quad \bar{Y}_0 = 0
  \end{align}
\end{subequations}
where $\mathcal{F}^{\bar{Y}}_t$ is the $\sigma$-algebra generated by
$(\bar{Y}_s)_{s\leq t}$. 

The conditional measure $\PP(\cdot | \mathcal{F}^{\bar{Y}}_T)$ has a
regular conditional probability density
$u \mapsto \PP(\cdot | \bar{Y} = u)$, $u \in C([0, T], \R)$. Now
observe that $B^X$ remains unchanged under $\PP(\cdot | \bar{Y} = u)$
since $\bar{Y}$ and $B^X$ are independent. This means that for all
$u \in C([0, T], \R)$ and $f \in C_c^\infty(\R)$,
$\PP(\cdot | \bar{Y}=u)$ solves the same martingale problem as the
measure associated to
\begin{equation}\label{eq:Xu}
  dX^u_t = -\alpha \kappa_X(X^u_t - u(t))dt + \sqrt\alpha
                 \sigma_X dB^X_t, \quad X^u_0 = 0.
\end{equation}
It follows that the conditional expectation given
$\mathcal{F}^{\bar{Y}}_T$ of any functional involving $\bar{X}$ equals
the usual expectation of the same functional with $\bar{X}$ replaced
by $X^u$ evaluated at $u = Y$.

For example, since
\begin{equation*}
  \E X^u_t = \int_0^t \alpha \kappa_X e^{-\alpha\kappa_X (t-s)} \, u(s) \, ds
\end{equation*}
the drift coefficient of $\bar{Y}$ is
\begin{multline*}
  \E\left[-\kappa_Y (\bar{Y}_t - \bar{X}_t) \middle|
                 \mathcal{F}^{\bar{Y}}_t \right] = -\kappa_Y
               (\bar{Y}_t - \E[\bar{X}_t | \mathcal{F}^{\bar{Y}}_T]) =
                -\kappa_Y
               (\bar{Y}_t - \E X^u_t|_{u = \bar{Y}}) \\ = -\kappa_Y
               \left(\bar{Y}_t - \int_0^t \alpha \kappa_X e^{-\alpha\kappa_X (t-s)} \, \bar{Y}_s \, ds\right)
\end{multline*}
so that $\bar{Y}$ solves the SDE
\begin{subequations}\label{eq:YbarZ}
  \begin{align}
    dZ_t &= -\alpha \kappa_X (Z_t - \bar{Y}_t) dt \\
    d\bar{Y}_t &= -\kappa_Y (\bar{Y}_t - Z_t) dt + \sigma_Y dB^Y_t.
  \end{align}
\end{subequations}

The key step in our estimate for $Y_t - \bar{Y}_t$ is
the application of the results from the first section to 
\begin{equation*}
  \int_0^T h(t)( X^u_t - \E X^u_t) dt
\end{equation*}
for a certain function $h(t)$.

We begin with a gradient estimate for the evolution operator
$P^u_{s,t}$ associated to $X^u$.
\begin{lemma} Let $\id(x) = x$ be the identity function and $h(t) \in
  C([0, T], \R)$. We have for all $x \in \R$
  \begin{equation*}
    \partial_x P^u_{s,t} (h \id) (x) = h(t) e^{-\alpha \kappa_X (t-s)}.
  \end{equation*}
\end{lemma}
\begin{proof}
Denote $X^{s,x}_t$ the solution to~\eqref{eq:Xu} with
$X^u_s = x$.
Then
\begin{equation*}
  d(X^{s,x+\varepsilon}_t - X^{s,x}_t) = -\alpha
  \kappa_X(X^{s,x+\varepsilon}_t - X^{s,x}_t) dt
\end{equation*}
so that
\begin{equation*}
  X^{s,x+\varepsilon}_t - X^{s,x}_t = \varepsilon e^{-\kappa_X \alpha (t-s)}
\end{equation*}
and
\begin{equation*}
  \partial_x P_{s,t} (h \id) (x) = \lim_{\varepsilon \to 0}
  \varepsilon^{-1} \E\left[h(t) X^{s,x+\varepsilon}_t - h(t) X^{s,x}_t\right] = h(t) e^{-\kappa_X \alpha (t-s)}.
\end{equation*}
\end{proof}

\begin{theorem}\label{theorem:multiou}
  \begin{align}\label{eq:equal1}
    \E\abs{Y_T - \bar{Y}_T}^2 &=
    \frac{\alpha \kappa_Y^2 \sigma_X^2}{(\alpha \kappa_X + \kappa_Y)^2} \int_0^T
    \left(1 - e^{-\alpha \kappa_X (T-t)}\left(2 - e^{-\kappa_Y(T-t)}\right)
    \right)^2 dt \\ &\leq \frac{T}{\alpha} \frac{\kappa_Y^2
                      \sigma_X^2}{\kappa_X^2} \notag 
  \end{align}
  and
  \begin{align}\label{eq:equal2}
    \E\abs{\bar{Y}_T - \sigma_Y B^Y_T}^2 &= \frac{\kappa_Y^2 \sigma_Y^2}{(\alpha
    \kappa_X + \kappa_Y)^2} \int_0^T \left(1 - e^{-(\alpha\kappa_X +
    \kappa_Y)t}\right)^2 dt \\ & \leq
\frac{T}{\alpha^2} \frac{\kappa_Y^2 \sigma_Y^2}{\kappa_X^2}. \notag
  \end{align}
\end{theorem}

\begin{proof}[Proof of Theorem~\ref{theorem:multiou}]
We now proceed to show the equality~\eqref{eq:equal1}. We decompose
\begin{align}\label{eq:YminusYbar}
  Y_T - \bar{Y}_T 
  &= \int_0^T \kappa_Y (X_t - Y_t) dt - \int_0^T \kappa_Y
                       (\E[\bar{X}_t | \bar{Y}] - \bar{Y}_t)dt \notag
  \\ 
  & = -\kappa_Y \int_0^T (\E[\bar{X}_t | \bar{Y}] - \bar{X}_t) dt 
    - \kappa_Y \int_0^T (Y_t - \bar{Y}_t) - (X_t -
  \bar{X}_t) dt.
\end{align}
Using linearity, we now proceed to rewrite this as
\begin{equation*}
  Y_T - \bar{Y}_T = -\kappa_Y \int_0^T h(T-t) (\E[\bar{X}_t | \bar{Y}] - \bar{X}_t) dt 
\end{equation*}
for some function $h$.

Since
\begin{equation*}
  d(X_t - \bar{X}_t) = -\alpha \kappa_X (X_t - \bar{X}_t) dt +
  \alpha\kappa_X (Y_t - \bar{Y}_t) dt
\end{equation*}
we have
\begin{gather*}
  X_t - \bar{X}_t = \int_0^t \alpha \kappa_X e^{-\alpha \kappa_X
    (t-s)} (Y_t - \bar{Y}_t) ds.
\end{gather*}
With the notation
\begin{equation*}
  f(t) = Y_t -
  \bar{Y}_t, \quad
  g(t) = \bar{X}_t - \E[\bar{X}_t | \bar{Y}]
\end{equation*}
equation~\eqref{eq:YminusYbar} reads as
\begin{equation*}
  \frac1{\kappa_Y} f'(t) + f(t) - \int_0^t \alpha \kappa_X e^{-\alpha \kappa_X
    (t-s)} f(s) ds = g(t).
\end{equation*}

Using capital letters for the Laplace transform, this writes as
\begin{equation*}
  \frac{s}{\kappa_Y} F(s) + F(s) - \frac{\alpha \kappa_X}{s + \alpha \kappa_X} F(s)= G(s)
\end{equation*}
or, after rearranging,
\begin{equation*}
  F(s) = \kappa_Y \frac{s+\alpha\kappa_X}{s(s + \alpha\kappa_X +
    \kappa_Y)} G(s) = \kappa_Y H(s)G(s).
\end{equation*}
Inverting the Laplace transform, we find that
\begin{gather*}
  h(t) = \frac{\alpha \kappa_X}{\alpha \kappa_X + \kappa_Y} +
  \frac{\kappa_Y}{\alpha \kappa_X + \kappa_Y} e^{-(\alpha\kappa_X + \kappa_Y)t}
\end{gather*}
so that
\begin{align*}
  Y_T - \bar{Y}_T &= \kappa_Y \int_0^T h(T-s)
  \left(\bar{X}_s - \E[\bar{X}_s | \bar{Y}] \right) ds.
\end{align*}

By the properties of conditional expectation and
Corollary~\ref{corollary:1} we have for any integrable function $\Phi$ that
\begin{equation*}
  \E \Phi(Y_T - \bar{Y}_T) = \E[ \E \Phi(Y_T - \bar{Y}_T) |
  \mathcal{F}^{\bar{Y}}_T] = \E[ \E(\Phi(Y_T - \bar{Y}_T) | u =
  \bar{Y}) ]
  = \E [(\E \Phi(M^u_T))|_{u = \bar{Y}}]
\end{equation*}
with
\begin{align*}
  M^u_T 
  &= \kappa_Y \int_0^T \int_t^T \partial_x P^u_{t,s} (h(T-\cdot)
  \id) (X_t) \, ds \, \sqrt{\alpha} \sigma_X dB_t \\
  &= \kappa_Y
    \sqrt{\alpha} \sigma_X \int_0^T \int_t^T h(T-s)
    e^{-\alpha \kappa_X (s-t)} \, ds \, dB_t \\
  &= \frac{ \kappa_Y \sqrt{\alpha} \sigma_X}{\alpha \kappa_X +
    \kappa_Y} \int_0^T \int_t^T \alpha \kappa_X e^{-\alpha \kappa_X (s -
    t)} ds + \int_t^T \kappa_Y e^{-\kappa_Y (T-s)} e^{-\alpha \kappa_X
    (T-s)} e^{-\alpha \kappa_X (s-t)} \, ds \, dB_t \\
  &= \frac{\sqrt{\alpha} \kappa_Y \sigma_X}{\alpha \kappa_X +
    \kappa_Y} \int_0^T 1 -
    e^{-(\alpha \kappa_X + \kappa_Y)(T-t)} \, dB_t.
\end{align*}
Since $M^u_t$ is independent of $u$ we can let $M_t =
M^u_t$ for an arbitrary $u$ so that
\begin{equation*}
  \E \Phi(Y_T - \bar{Y}_T) = \E \Phi(M_T).
\end{equation*}

Now we can compute
\begin{align*}
  \E \Abs{Y_T - \bar{Y}_T}^2 = \E \QV{M}_T = \frac{\alpha \kappa_Y^2 \sigma_X^2}{(\alpha \kappa_X +
    \kappa_Y)^2} \int_0^T \left(1 - 
    e^{-(\alpha \kappa + \kappa_Y)(T-t)} \right)^2 dt.
\end{align*}

We now turn to the computation of $\E\abs{\bar{Y}_t - \sigma_Y B^Y_t}^2$.

From equation~\eqref{eq:YbarZ} we have
\begin{equation*}
  d (\bar{Y}_t - Z_t) = -(\alpha \kappa_X + \kappa_Y) (\bar{Y}_t -
  Z_t) dt + \sigma_Y B^Y_t
\end{equation*}
so that
\begin{equation}\label{eq:3}
  \bar{Y}_t - Z_t = \sigma_Y \int_0^t e^{-(\alpha \kappa_X +
    \kappa_Y)(t-s)} dB^Y_s.
\end{equation}
is an Ornstein-Uhlenbeck process. This means that
\begin{equation*}
  \E (\bar{Y}_t - Z_t) (\bar{Y}_s - Z_s) = \frac{\sigma_Y^2 e^{-(\alpha\kappa_X + \kappa_Y)t}}{\alpha
    \kappa_X + \kappa_Y} \sinh((\alpha \kappa_X +
  \kappa_Y)s), \quad s \leq t.
\end{equation*}
so that
\begin{align*}
  \E\abs{\bar{Y}_t - \sigma_Y B^Y_t}^2 
  &= \kappa_Y^2 \Abs{\int_0^t
  \bar{Y}_s - Z_s ds}^2  \\
  &= 2 \kappa_Y^2 \int_0^t \int_0^s \E (\bar{Y}_s
    - Z_s) (\bar{Y}_r - Z_r) dr ds \\
  &= \frac{2 \kappa_Y^2 \sigma_Y^2}{(\alpha
    \kappa_X + \kappa_Y)} \int_0^t e^{-(\alpha\kappa_X + \kappa_Y)s} \int_0^s 
    \sinh((\alpha \kappa_X + \kappa_Y) r) dr ds \\
  &= \frac{2 \kappa_Y^2 \sigma_Y^2}{(\alpha
    \kappa_X + \kappa_Y)^2} \int_0^t e^{-(\alpha\kappa_X + \kappa_Y)s}
    \left(\cosh((\alpha \kappa_X + \kappa_Y) s) - 1 \right) ds \\
  &= \frac{\kappa_Y^2 \sigma_Y^2}{(\alpha
    \kappa_X + \kappa_Y)^2} \left(\int_0^t 1 + e^{-2 (\alpha\kappa_X +
    \kappa_Y)s} - 2 e^{-(\alpha\kappa_X + \kappa_Y)s} ds \right)
\end{align*}
\end{proof}

\section{Approximation by Averaged Measures}
In the previous section, the computation for
$\E\abs{\bar{Y}_t - \sigma_Y B^Y_t}^2$ relied on the fact that we had
an explicit expression for $\E[\bar{X}_t - \bar{Y}_t | Y]$. Here we will see a
method that can be used to obtain similar estimates in more general
situations.

Consider a diffusion process $(X_t, Y_t)$ on $\R^n \times \R^m$
\begin{align*}
  dX_t &= b_X(X_t, Y_t) dt + \sigma_X(X_t, Y_t) dB^X_t \\
  dY_t &= b_Y(Y_t) dt + \sigma_Y(Y_t) dB^Y_t
\end{align*}
where $B^X$ and $B^Y$ are standard independent Brownian
motions. Denote $L$ the generator of $(X, Y)$ and $\mathcal{F}^Y$ the filtration of $B^Y$.

\newcommand{\EFy}[1]{\E^{\mathcal{F}^Y_#1}}

Let
\begin{equation*}
    Q_t f = \EFy t f(X_t, Y_t)
\end{equation*}
so that, by the Itô formula and since $Y$ is adapted to
$\mathcal{F}^Y$ and $B^X$ and $B^Y$ are independent, we have
\begin{align*}
  Q_t f
  &= \EFy{t} \Biggl[f(X_0, Y_0) + \int_0^t L f(X_s, Y_s) ds +
    \int_0^t \nabla_x f(X_s, Y_s) \cdot \sigma_X(X_s, Y_s) dB^X_s \\ &\quad +
    \int_0^t \nabla_y f(X_s, Y_s) \cdot \sigma_Y(Y_s) dB^Y_s \Biggr]
  \\
  &= \EFy{0}[f(X_0, Y_0)] + \int_0^t \EFy s Lf(X_s, Y_s) ds + \int_0^t
    (\EFy s \nabla_y f(X_s, Y_s)) \cdot \sigma_Y(Y_s) dB^Y_s.
\end{align*}

In other words,
\begin{align*}
  dQ_t f &= Q_t L f dt + (Q_t \nabla_y f) \cdot
           \sigma_Y(Y_t) dB^Y_t.
\end{align*}

\begin{example}[Averaged Ornstein-Uhlenbeck]
Consider again the process $(\bar{X}, \bar{Y})$ from the previous
section. In this case, $f(x, y) = x - y$ is an eigenfunction of $-L$
with eigenvalue $\alpha \kappa_X + \kappa_Y$ and we have $\partial_y f
= -1$. Therefore
\begin{equation*}
  dQ_t f = -(\alpha \kappa_X + \kappa_Y) Q_t f dt - \sigma_Y dB^Y_t
\end{equation*}
so that we retrieve the result from~\eqref{eq:3}
\begin{equation*}
   \E[\bar{X}_t - \bar{Y}_t | \bar{Y}] = Q_t f = -\sigma_Y \int_0^t e^{-(\alpha
     \kappa_X + \kappa_Y)(t-s)} dB^Y_s.
\end{equation*}

\end{example}

\printbibliography

\end{document}